%% file: RandomJigSaw7.tex
\documentclass[11pt,oneside,reqno]{article}

\usepackage{amsmath,amsthm,amssymb,color}
\usepackage{bbm}
\usepackage{mathrsfs,enumerate}
\usepackage[breaklinks=true]{hyperref}
\usepackage{textcase}
\usepackage{pgfplots}
\usepackage{float}
\usepackage{graphicx}
\usepackage{breakcites}

\numberwithin{equation}{section}
\allowdisplaybreaks[4]

\theoremstyle{plain}
\newtheorem{theorem}{Theorem}[section]
\newtheorem{proposition}[theorem]{Proposition}
\newtheorem{lemma}[theorem]{Lemma}

\theoremstyle{definition}
\newtheorem{definition}[theorem]{Definition}
\newtheorem{example}[theorem]{Example}



\def\eps{\varepsilon}

\newcommand{\bz}{{\mathbf Z}}

\newcommand{\dZ}{\bz}
 
\newcommand{\PAR}[1]{{{\left(#1\right)}}}

\newcommand{\INT}{\mathrm{in}}

\begin{document}

\title{\sc\bf\large\MakeUppercase{
Shotgun Assembly of Random Jigsaw Puzzles}}
\author{
Charles Bordenave\thanks{CNRS and Universit\'e Toulouse III; \href{mailto:bordenave@math.univ-toulouse.fr}{bordenave@math.univ-toulouse.fr}} 
\and
Uriel Feige\thanks{Weizmann Institute; \href{mailto:uriel.feige@weizmann.ac.il}{uriel.feige@weizmann.ac.il}}
\and
	Elchanan Mossel\thanks{University of California, Berkeley and University of Pennsylvania; \href{mailto:mossel@wharton.upenn.edu}{mossel@wharton.upenn.edu. Research supported by NSF grant
CCF-1320105, DOD ONR grant N00014-14-1-0823, and Simons Foundation grant 328025.}}
}
\date{\today}
\maketitle

\begin{abstract}
In a recent work, Mossel and Ross considered the shotgun assembly problem for a random jigsaw puzzle. Their model consists of a puzzle - an $n\times n$ grid,  where each vertex is viewed as a center of a piece. They assume that each of the four edges adjacent to a vertex, is assigned one of $q$ colors (corresponding to "jigs", or cut shapes) uniformly at random.
Mossel and Ross asked: how large should $q = q(n)$ be so that with high probability the puzzle can be assembled uniquely given the collection of individual tiles?
They showed that if $q = \omega(n^2)$, then the puzzle can be assembled uniquely with high probability, while if $q = o(n^{2/3})$, then with high probability the puzzle cannot be uniquely assembled. Here we improve the upper bound and show that for any $\eps > 0$, the puzzle can be assembled uniquely with high probability if $q \geq n^{1+\eps}$.
The proof uses an algorithm of $n^{\Theta(1/\eps)}$ running time.
\end{abstract}

\section{Introduction}

\cite{MR15} recently suggested the following problem:
Consider a factory that manufactures jigsaw puzzles. 
The factory aims to make sure that a unique assembly of the puzzle is guaranteed just from the way the pieces are cut, regardless of whether the images on the puzzle are informative (e.g., even if there is a large patch of sky).
Suppose that there are $q$ different type of jigs (cut shapes between adjacent pieces), that the puzzle is of size $n \times n$, and that the type of jig between any two adjacent pieces is selected at random. {\em How large should $q$ be so that a random puzzle drawn from this distribution has unique assembly?}
This problem, which they called "shotgun assembly of random jigsaw puzzle",  
 is a two dimension variant of the well studied problem of shotgun assembly of DNA sequences, which is extensively studied from both the combinatorial and probabilistic view points, see e.g. ., \cite{Arratia1996}, \cite{Dyer1994}, and \cite{Motahari2013} 

Let us present the above question in a formal manner where the puzzle will be defined as the $n$ by $n$ grid graph with a uniform $q$ coloring of the edges of the grid. From now on we will use the graph theoretic notion of color instead of jig (cut shape, also referred to as``knobs",``locks",``tabs", ``slots", ``indents" etc. in the jigsaw puzzle terminology). 
The parameters for our model are two positive integers, $n$ and $q$. We use the notation $[m]$ to denote the set of numbers $\{1, \ldots, m\}$, and $[a,b]$ to denote the set $\{a,a+1,\ldots,b-1,b\}$. A puzzle may be thought of as an $n$ by $n$ grid with colored edges. The building blocks of the puzzle are {\em pieces}  - i.e., vertices of the grid along with $4$ adjacent colored half edges.  
Observe that every vertex not on the boundary of the grid is incident with exactly~4 edges. We assume for simplicity of the presentation (this will not significantly effect the results in the current manuscript) that also every vertex on the boundary is incident with~4 edges. This involves introducing boundary edges that lead out of the grid and do not have vertices at their other endpoint. We further assume for simplicity  that at any given vertex $v$, the edges incident with it are labeled by their orientation: Up, Down, Right and Left and denoted $\uparrow(v)$, $\downarrow(v)$, $\rightarrow(v)$ and $\leftarrow(v)$. We denote by $\sigma$ the coloring, so that the colors incident to $v$ are 
$\sigma(\uparrow(v))$, $\sigma(\downarrow(v))$, $\sigma(\rightarrow(v))$ and $\sigma(\leftarrow(v))$. 
Each edge (including the boundary edges) is given a random color in $[q]$ (corresponding to the type of jig being used), uniformly at random and independently across edges. Thereafter, the puzzle is disassembled, and its pieces are presented at a random order. At this point, the input is $n^2$ pieces, where each piece is a vertex with~4 incident edges labeled as Up, Down, Right and Left, and colored by colors from $[q]$. An {\em assembly} of the pieces is a placement of the vertices on an $n$ by $n$ grid, where for each vertex the edges are oriented in the direction of their labels. The assembly is {\em feasible} if for every two adjacent vertices the colors that they have for their common edge are the same. We refer to the assembly that gives back the original puzzle as the {\em planted assembly}.

We say that a puzzle has {\em unique vertex assembly} if it has only one feasible assembly, namely, the planted assembly. We say that a puzzle has {\em unique edge assembly} if for every feasible assembly and for every edge location (not including boundary edges), the color of the respective edge is the same as in the planted assembly.
Note that a puzzle with two identical pieces will not have unique vertex assembly, but it may have unique edge assembly.

Since the probability of having each type of piece is $q^{-4}$, it follows by the birthday paradox that two identical pieces exist with high probability as soon as $q = o(n)$, and in this case the puzzle does not have unique vertex assembly. It is further shown in~\cite{MR15} that if $q = o(n^{2/3})$ then with high probability a random puzzle will not have unique edge assembly.
\cite{MR15}~ further provided a linear time algorithm for unique vertex assembly when
$q \geq C n^2$ for a sufficiently large constant $C$.

One of the main open problem of~\cite{MR15} was to obtain more accurate bounds for the jigsaw assembly problem. Here we improve the upper bound by proving the following:

\begin{theorem}
\label{thm:main}
For every $\eps > 0$, if $q \geq n^{1+\eps}$ then with high probability a random puzzle has unique vertex assembly. Moreover, there is an algorithm running in time $n^{O(1/\eps)}$ that with high probability finds the planted assembly.
\end{theorem}

Here and elsewhere, the expression ``with high probability" means with probability going to $1$ as $n \to \infty$. We will write $C(k)$ for a constant depending on $k$ only. The value of $C(k)$ at different occurrences will be different. 

The proof of Theorem~\ref{thm:main} is based on the following principle. For a given integer parameter $k > 1$ (where $k$ is a constant independent of $n$), we refer to a $2k+1$ by $2k+1$ grid as a {\em window}, and index it by $[-k,k] \times [-k,k]$. Given an input of $n^2$ pieces,
for each piece $v$,  we consider all possible sets of $(2k + 1)^2$ pieces (including $v$ itself) and check if they can be assembled as a feasible (namely, legally colored) window with $v$ at its center. A feasible assembly of a window with $v$ at its center will be referred to as a $v$-window. Given a $v$-window, the neighborhood $\{ (0,\pm 1), (\pm 1, 0)\}$ of $v$ in the $v$-window is considered to be a {\em candidate neighborhood} (or in more details, an ${\ell}_{1}$ radius $1$ candidate neighborhood) of $v$ in the puzzle.

For every vertex $v$ there might be several different $v$-windows, and hence several candidate neighborhoods. Nevertheless, for a choice of $k = O(1/\eps)$ we show that with high probability for every vertex  at distance at least $k + 1$ from the boundary of the puzzle, its $\ell_1$ radius $1$ candidate neighborhood is unique.  Consequently, this rigidity allows us to assemble the part of the puzzle at distance $k+1$ from the boundaries of the puzzle. A simple algorithm then allows to assemble the rest of the puzzle.

The paper is organized as follows. In Section \ref{sec:loc}, we formalize the above notion and state our main result on the $v$-window. In Section \ref{sec:constraint}, we translate in graphical terms the problem of feasibility of an assembly. Section \ref{sec:iso} contains our isoperimetric analysis and Section \ref{sec:algo} describes the reconstruction algorithm.  Finally, Section \ref{sec:fin} discusses the extension where jigs have shapes instead of colors and can be rotated. 

\section{Local Assembly}
\label{sec:loc}
For vertex $v \in [n]^2$, let $S_k(v)$ denote the set of injective maps $f : [-k,k]^2 \to [n]^2$ such that
\begin{itemize}
\item
$f(0,0) = v$ and
\item $( f(i,j) : -k \leq i \leq k, -k \leq j \leq k )$ is {\em feasible}, that is
$\sigma(\rightarrow(f(i,j))) = \sigma(\leftarrow(f(i+1,j)))$ for all $i,j$ s.t. $(i,j),(i+1,j) \in [-k,k]^2$,
and $\sigma(\uparrow(f(i,j))) = \sigma(\downarrow(f(i,j+1)))$ for all $i,j$ s.t. $(i,j),(i,j+1) \in [-k,k]^2$
\end{itemize}
Note that $S_k(v)$ may be empty if $v$ is of distance less than $k$ from the boundaries of the grid. Otherwise, $S_k(v)$ contains at least one element, namely the one given by $f(x) = v + x$ for all $x \in [-k,k]^2$.

The main theorem we wish to prove is the following:

\begin{theorem} \label{thm:k_bound}
There exists $c > 0$ such that for
all $\eps > 0$, if $k \geq c/\eps$ then the following holds:
For every $v \in [n]^2$ and for every
$\alpha \in \{ (0, \pm 1), (\pm 1, 0) \}$
\[
P[ \exists f \in S_k(v) \mbox{ s.t. } f(\alpha) \neq v + \alpha] \leq C(k) n^{-2-\eps/2}.
\]
\end{theorem}

Theorem~\ref{thm:k_bound} is the main result needed to prove that with high probability 
 all vertices at distance at most $k$ from the boundaries can be assembled correctly. A simple algorithm then allows to construct the reminder of the puzzle. This will allow us to establish  Theorem \ref{thm:main}.


\section{The Constraint Graph}
\label{sec:constraint}
The proof of Theorem \ref{thm:k_bound} is based on a detailed analysis of the constraints imposed by the condition that an injective function $f [-k,k]^2 \to [n]^2$ is feasible, along with isoperimetric reasoning in order to lower bound the number of constraints.

To simplify notation we write $(i,j : j+1)$ for the edge $((i,j),(i,j+1))$. 
 Similarly we write $(i:i+1,j) := ((i,j),(i+1,j))$. 
Note that by definition
\[
\rightarrow(i,j) = (i:i+1,j) = \leftarrow(i+1,j), \quad
\uparrow(i,j) = (i,j:j+1) = \downarrow(i,j+1).
\]

Sometimes it would be more useful to analyze the constraints imposed by $f$ on a subset of $[-k,k]^2$. This leads to the following definitions:

\begin{definition}
For a given $f : [-k,k]^2 \to [n]^2$, and $W \subset [-k,k]^2$, the {\em restriction of $f$ to $W$},
denoted $f_{|W}$, is the function $f_{|W} : W \to [n,n]^2$ defined by $f_{|W}(w) = f(w)$, for all
$w \in W$.

Given $f : [-k,k]^2 \to [n]^2$ and $W \subset [-k,k]^2$, the {\em tiles} of $f_{|W}$, denoted $T(f_{|W})$ is the collection of connected components of the graph with vertex set $f(W)$ and where vertices $v,w$ are adjacent if $v - w \in \{\pm (0,1), \pm (1,0) \}$.
We write $T(f)$ for $T(f_{| [-k,k]^2})$ and call $T(f)$ the tiles of $f$. 
\end{definition}
Note that the tiles are defined in terms of the image of the map $f$. 


\begin{definition}
The constraint graph $G(f_{|W}) = (V,E)$ of $f_{|W}$ for $f : [-k,k]^2 \to [n]^2$
is the graph whose  whose edge set $E$ consists of 
\begin{align*}
(\rightarrow(f(u)),\leftarrow(f(u+(1,0))) &\,&  \quad \mbox{ if } f(u + (1,0)) \neq f(u) + (1,0)
&\,& \mbox{ and } u,u+(1,0) \in W,
\\
(\uparrow(f(u)),\downarrow(f(u+(0,1))) &\,&  \quad \mbox{ if } f(u + (0,1)) \neq f(u) + (0,1)
&\,& \mbox{ and } u,u+(0,1) \in W.
\end{align*}
The vertex set $V$ of $G(f_{|W}))$ is the set of all edges of $[n]^2$ spanned by $E$. 
The constraint graph of $f$ is the constraint graph of $f_{|W}$ for $W = [-k,k]^2$.
We write $c(f_{|W})$ for the number of connected
components of $G$ and $\gamma(f_{|W}) = |V| - c(f_{|W})$.
We will omit the subscript $W$ when $W = [-k,k]^2$.
\end{definition}

Consider a candidate $f : [-k,k]^2 \to [n]^2$. 
We say that an edge $((i:i+1,j), (i':i'+1,j'))$ of the constraint graph $G(f)$ is {\em satisfied} if $\sigma((i:i+1,j)) = \sigma((i':i'+1,j'))$ and similarly for an edge
$((i,j:j+1),(i', j' : j' +1))$. 
We say that $G(f_{|W})$ is {\em satisfied} if all of its edges are satisfied. 
To distinguish the vertices and edges of the grid from those of $G$, we will sometime write explicitly $G$-vertices and $G$-edges and 
grid-vertices and grid-edges.

\begin{lemma} \label{lem:constraint}
$f_{|W}$ is feasible iff  $G(f_{|W})$ is satisfied . Moreover, for a fixed $f : [-k,k]^2 \to [n]$ and $W \subset [-k,k]^2$, the probability that $f_{|W}$ is feasible for a random puzzle 
is $q^{-\gamma(f_{|W})}$.
\end{lemma}

\begin{proof}
The first statement follows from the definitions. For the second statement we will compute the probability that 
$G(f_{|W})$ is satisfied. For $G(f_{|W})$ to be satisfied, it is required that the color of $G$-vertices of $G(f_{|W})$ (grid-edges) in each connected component are identical. Note that events for different components are independent and the probability that a certain component $C$ has all $G$-vertices of the same color is $q^{-c+1}$ where $c$ is the number of $G$-vertices in $C$. 
The conclusion follows.
\end{proof}

Note that the degree of each $G$-vertex of is either $1$ or $2$. Therefore the connected components of $G(f_{|W})$ are either paths 
or cycles. 

\begin{example} \label{ex:local}
Let $W = [1,2] \times [1,2]$ and let $g = f_{|W}$ be defined by
\[
g(1,1) = (1,1), \quad g(1,2) = (3,2), \quad g(2,1) = (3,1), \quad g(2,2) = (1,2).
\]
In this case, the map $f_{|W}$ has $2$ tiles, namely \{(1,1),(1,2)\}, \{(3,1),(3,2)\}. The constraint graph is the graph with the following edges:
\[
((1,1:2),(3,1:2)), \quad ((1,1:2),(3,1:2)), \quad ((1:2,1),(2:3,1)), \quad ((3:4,2),(1 : 0,2))
\]
Note that the first edge is a double edge as it is imposed both by the adjacencies of
$(1,1)$ to the left of $(3,2)$ and of $(3,1)$ to the right of $(1,2)$.
The vertex set $V$ of $G(f)$ consists of
\[
(1,1:2),(3,1:2),(1:2,1),(2:3,1),(3:4,2),(1 : 0,2)
\]
and is of size $6$.
The connected components of $G(f_{|W})$ are given precisely by the $3$ edges.
Thus $|V| = 6$, the number of connect components is $3$ and the probability that $f_{|W}$ is feasible is $q^{3-6} = q^{-3}$.
see Figure \ref{fig:2} (right).
\end{example}

\begin{figure}[htb]
\centering \scalebox{0.6}{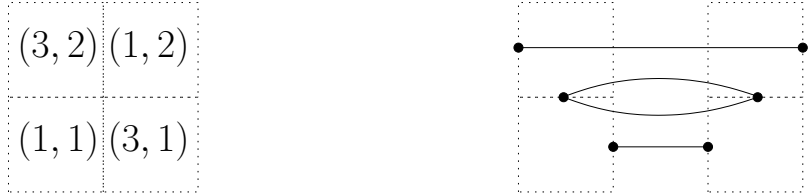}
\caption{The local assembly in Example \ref{ex:local} (left) and its constraint graph (right).}\label{fig:2}
\end{figure}

\begin{proposition} \label{prop:constraint}
 Let $u(f_{|W})$ denote the number of constraints of $G(f_{|W})$ containing a $V$-vertex 
  that appears once in all constraints and let $w(f_{|W})$ denote the total number of constraints.
Then $\gamma(f_{|W}) \geq u(f_{|W}) + 0.5 (w(f_{|W})-u(f_{|W})) = 0.5 w(f_{|W}) + 0.5 u(f_{|W})$.
\end{proposition}

\begin{proof}
As noted earlier the degree of each vertex in $V$ is at most two. Therefore the graph $G(f_{|W})$ is a disjoint union of cycles and paths. 
Moreover, $\gamma, u$ and $w$ are all additive over disjoint components. Therefore it suffices to check the claim for paths and cycles of length at least $2$. 
For a path of length $2$ we have $\gamma = u = w = 1$ as needed and for a path of length $\ell \geq 3$: 
\[
\gamma = \ell - 1, u = 2, w = \ell-1 
\]
so the inequality holds in this case as well. The case of cycles is even simpler since for a cycle of length $\ell \geq 2$ we have: 
\[
\gamma = \ell-1, u = 0, w = \ell
\]
\end{proof}

The proof of Theorem \ref{thm:k_bound} is based on isoperimetric results proved in the next section. For a subset $T_i$ of $[n]^2$ we let $\partial T_i$ denote the {\em edge boundary} of $T_i$ and $|\partial T_i|$ denote the length of the boundary, i.e., the number of edges between $T_i$ and its complement.

\begin{lemma} \label{lem:iso_main}
Let $S_k'(v)$ denote the subset of $S_k(v)$ where there exists an
$ \alpha \in \{\pm (0,1), \pm (1,0) \}$ with $f(\alpha) \neq v + \alpha$.
For $f \in S_k'(v)$ let $T = T(f)$ be the collection of tiles in $[-n,n]^2$ determined by
$f$.
Then for every $\eps > 0$ if $k > c/\eps$ for a large enough $c$ then the following holds.
For every $f \in S_k'(v)$, there exists a $W$ such that
$0 \in W \subset [-k,k]^2$ with the following property.
Let $t + 1 = |\{ i : W \cap T_i \neq \emptyset \}|$. Then
$(1+\eps) \gamma(f_{|W}) \geq 2t + 2 + \eps$.
\end{lemma}

We now prove Theorem \ref{thm:k_bound} assuming Lemma~\ref{lem:iso_main}.

\begin{proof}[Proof of Theorem \ref{thm:k_bound}].
We want to bound the probability that there exists a feasible $f$ 
where $f : [-k,k]^2 \to [n]^2$ with $f(0) = v$ and 
$f(\alpha) \neq v+\alpha$ for some $\alpha \in \{ \pm (0,1), \pm (1,0) \}$. 
By Lemma~\ref{lem:iso_main} is suffices to bound for each $W \subset [-k,k]^2$ with 
$0 \in W$, the probability that there exists such $f$ where $f_{|W}$ is feasible and moreover 
$(1+\eps) \gamma(f_{|W}) \geq 2t + 2 + \eps$. 

Note that the number of choices of $W$ is $C(k)$. Given $W$ and the fact that
$f(0) = v$, the number of choices of $f_{|W}$ is at most $C(k) n^{2t}$.
This follows since each tile $T$ is determined by one $f(w) \in T$ and a subset of
$S \subset [-k,k]^2$.

By Lemma~\ref{lem:constraint}, the probability that $f_{|W}$ is feasible is bounded above by
$q^{-\gamma(f_{|W})}$ which can be bounded by $n^{-2t-2-\eps}$ by Lemma~\ref{lem:iso_main}.

Since $t \leq (2k+1)^2$ it follows that
the overall probability that such an $f$ exists with $f(0) = v$ is upper bounded by
$C(k) n^{-2-\eps}$ as needed.\end{proof}

\section{Isoperimetric Analysis}
\label{sec:iso}
In this section, we will prove the main isoperimetric lemma, i.e.
Lemma~\ref{lem:iso_main}. We start by proving the following lemma:

\begin{lemma} \label{lem:iso_simple}
Let $f \in [-k,k]^2 \to [n]^2$ with the number of tiles in $f$, $|T(f)| = t+1 \geq 2$.
Then
\[
\gamma(f) \geq t(2 - \frac{2}{s}), 
\]
where $s = 2k + 1$. Moreover, if two tiles have more than $35$ pieces then  
\[
\gamma(f) \geq  2  t (1 - \frac{1}{s} ) +4.
\]
\end{lemma}

Our proof will be based on the following classical fact.

\begin{lemma} \label{lem:classic}
Let $A \subset R^2$ be a set with boundary that is axis aligned. Then the length of its boundary $\partial A$ satisfies $|\partial A| \geq 4 |A|^{1/2}$, where $|A|$ is the area of the set.
\end{lemma}

A special case of the lemma above is the elementary exercise showing that the square minimizes the surface area among all rectangles of a given area. The more general case can be proved for example by looking at the minimal axis align rectangle containing the body $A$ and observing that its surface area must be smaller or equal to the surface area of $A$. 
The following lemma will be used in the proof of Lemma~\ref{lem:iso_simple}.

\begin{lemma} \label{lem:iso2}
If $a_0 \geq a_1 \geq \cdots a_t \geq 1$ is an integer partition of $s^2$, $\sum_i a_i = s^2$, let 
\[
g =   2\sum_{i=0}^t  \sqrt{a_i} -  2 s.
\]
Then
\[
g \geq 2t(1 - \frac{1}{s}).
\]
Moreover if $a_0 \geq a_1 \geq 36$ then 
\[
g \geq 2t(1 - \frac{1}{s}) + 4.
\]
\end{lemma}

\begin{proof}
Since $x \to x^{1/2}$ is concave, the minimum of $g$ under the constraints that $\sum a_i = s^2$ and each $a_i \geq 1$ is obtained when all of the $a_i$ but one,
satisfy $a_i = 1$. Thus
\[
g \geq 2 (t + s \sqrt{1-t/s^2} - s)  \geq 2(t + s(1-t/s^2) - s) =
2 t (1 - \frac{1}{s}). 
\]
The first statement proof follows. When $a_0 \geq a_1 \geq 36$, utilizing the concavity of $x^{1/2}$ allows to obtain a better bound. consider the integer partition $b$ obtained by joining all the mass of $a_1$ to $a_0$ except one unit that is left separately:
\[
b_0 = a_0 + a_1 - 1,\quad b_1 = a_2,\ldots,b_{t-1} = a_{t}, \quad b_ t = 1.
\]
 Since $\sqrt{a_0} \geq \sqrt{a_1} \geq 6$  we get 
 \[
10 + 2 \sqrt{a_0} \sqrt{a_1} \geq 2 \sqrt{a_0} \sqrt{a_1}  \geq 6 (\sqrt{a_0} + \sqrt{a_1})
\]
This implies 
\[
(\sqrt{a_0} + \sqrt{a_1} - 3)^2 =  a_0 + a_1 + 9 - 6 (\sqrt{a_0} + \sqrt{a_1}) + 2 \sqrt{a_0} \sqrt{a_1} \geq a_0 + a_1 -1, 
\]
so taking square roots we see that 
\[
 \sqrt{a_0} + \sqrt{a_1} \geq \sqrt{a_0 + a_1 - 1} + 3 =  \sqrt{b_0} + \sqrt{b_t} + 2. 
\]
Hence, the first statement of the lemma gives
$$
2 \sum_{i=0}^t \sqrt a_i - 2s \geq 2 \sum_{i=0}^t \sqrt b_i - 2 s+ 4 \geq 2t (1 - \frac{1}{s}) +  4,
$$
as needed. 
\end{proof}

We can now prove Lemma~\ref{lem:iso_simple}
\begin{proof}[Proof of Lemma ~\ref{lem:iso_simple}]
Note that except for the edges at the boundary of the grid $[-k,k]^2$, every edge at the boundary of one of the tiles $T_0,\ldots,T_t$ is part of a constraint and appears uniquely.
Thus by Proposition~\ref{prop:constraint} it follows that
\[
\gamma(f) \geq  \frac{1}{2} (\sum_{i=0}^t |\partial T_i| - 4 s)  \geq 2 \sum_{i=0}^{t} \sqrt{|\partial T_i| } - 2 s. 
\]
where the second inequality follows from  Lemma~\ref{lem:classic}. The lemma is then a consequence  of Lemma~\ref{lem:iso2} 
\end{proof}

We now prove Lemma~\ref{lem:iso_main}.

\begin{proof}[Proof of Lemma \ref{lem:iso_main}]
We will take $c = 200$ so $k \geq 200/\eps$.
We will consider a few cases. Let $\tau+1$ be the number of tiles of $f$.
\begin{itemize}
\item
 $\tau \geq 3/\eps$. In this case, we set $W = [-k,k]^2$.  Then $t = \tau$ and Lemma~\ref{lem:iso_simple} implies that
\[
\gamma(f)(1+\eps) \geq 2 t (1-1/k)(1+\eps) \geq 2t(1+\eps/2) \geq 2t + 3.  
\]
Hence, the set $W$ satisfies the conclusion of Lemma \ref{lem:iso_main}.
\item
We next consider the case where the second largest tile is
of area at least $36$. We may also take $W = [-k,k]^2$. Then $t = \tau$ and by Lemma~\ref{lem:iso2},
\[
\gamma(f)(1+\eps) \geq 2 t (1+\eps) (1 - \frac{1}{2k+1}) +  4 \geq 2t  + 4,
\]
as needed.

\item
%

We next consider the case where $f([-2,2]^2)$ is all part of the same tile of $f$. In this case, we set $W = [-2,2]^2$. Since $T(f_{|W}) = 1$, it is sufficient to check that $\gamma ( f_{|W|} ) \geq 2$. To this end, consider the graph $H$ with vertex set $W$ obtained by joining, for $\beta \in \{ \pm (1,0) , \pm (0,1) \}$, $x$ and $x+\beta$ if $f(x + \beta) = f(x) + \beta$. In words, the edges of $H$ correspond to pairs of vertices that are adjacent both in $W$ and in the original puzzle. Therefore, if $x$ and $y$ are in the same connected component of $H$ then 
$f(y) = f(x + (y-x)) = f(x) + (y-x)$ (this can be proven by induction on the length of the minimal path connecting $x$ and $y$ in $H$). Thus since $y-x \in \{ \pm (1,0) , \pm (0,1) \}$ it follows that $\{x, y\}$ is an edge of $H$. Hence, our assumption $f(0) = v$ and $f(\alpha ) \ne v + \alpha$ for some $\alpha \in  \{ \pm (1,0) , \pm (0,1) \}$, implies that $0$ and $\alpha$ are not in the same connected component of $H$.  On the other hand, we observe that except for the edges in  $\partial W$, every edge  at the boundary of a connected component of $H$ is part of a constraint in $G(f_{|W})$.  By inspecting the possible configurations of the connected component of $\alpha$ in $G$, we see that has a at least $4$ edges on its boundary which are not in $\partial W$. It follows there are at least $4$ constraints. By  Lemma~\ref{prop:constraint}, it implies that $\gamma(f_{|W} ) \geq 2$ as needed.

\item
The last case is where $\tau < 3/\eps$, all the parts but one are of area at most $36$ and there exist $y, x \in [-2,2]^2 $ which belong to different tiles. Let $T_0$ be the tile of $f$ with the maximal size. Note that
\[
|T_0| \geq (2k+1)^2 - 3/\eps*36 > (2k+1)*(2k).
\]
Since $f(x)$ and $f(y)$ lie in different tiles, at least one of the two doesn't belong to 
$T_0$. WLOG assume that $x \in f^{-1}(T_1)$ where
$|T_1| \leq 36$. Let $W'$ denote the connected component of $x$ in the subset
$[-k,k]^2 \setminus f^{-1}(T_0)$.
A key observation is that since $\tau \times 36 +2 < k = 200/\eps$, it follows that none of the elements of $W'$ are adjacent to the boundary of the grid $[-k,k]^2$.
In other words each edge in $\partial W'$ has one of its end point in $f^{-1}(T_0)$. This implies that $\partial_v W' \subset f^{-1}(T_0)$, where  $\partial_v W'$ is the {\em vertex boundary} of $W'$. We set $W = W' \cup \partial_v W'$.

Define $U_0 = f^{-1}(T_0)$ and let $U_i = f^{-1}(T_i) \cap W$. We assume without loss of generality
that $U_i \neq \emptyset$ for $i=0,\ldots,t$ and $U_i$ is empty otherwise.
In other words, the number of tiles of $f_{|W}$ is $t+1$. 
We wish to lower bound $\gamma(f_{|W})$.
Note that every edge between different $U_i$'s defines a constraint.
Thus
\[
w(f_{|W}) \geq \frac{1}{2}\sum_{i=1}^t |\partial U_i| + \frac{1}{2} |\partial W'|.
\]
Moreover, every edge in $\partial f(U_i)$ defines a constraint with a vertex that appears only once.
Thus
\[
u(f_{|W}) \geq \frac{1}{2} \sum_{i=1}^t |\partial f(U_i)|.
\]
Thus by Proposition~\ref{prop:constraint} and the fact that the boundary of each set is at least
$4$ it follows that
\[
\gamma(f_{|W}) \geq \frac{1}{2} (w(f) + u(f)) \geq \frac{1}{4}
\left(\sum_{i=1}^t |\partial U_i| + \sum_{i=1}^t |\partial f(U_i)| + |\partial W'| \right) \geq
2 t + \frac{1}{4} |\partial W'|.
\]
If $|W'| \geq 2$ then $|\partial W'| \geq 6$ and so $\gamma(f_{|W}) \geq 2 t + 1.5$. 
However since $\gamma(f)$ is integer we get $\gamma(f_{|W}) \geq 2t + 2$ and therefore
\[
\gamma(f_{|W})(1+\eps) \geq 2 t + 2 + \eps,
\]
as needed.
So it remains to prove the claim when $|W'|=1$. In this case, $\gamma(f_{|W}) = 4$ and $(1+\eps) \gamma(f_{|W}) \geq 4 + \eps$ as needed.
\end{itemize}
The proof is complete.
\end{proof}

\section{Algorithmic aspects} \label{sec:algo}

We now prove our main result Theorem \ref{thm:main}.  We will describe a deterministic algorithm which reconstructs the planted assembly with high probability if $q \geq n^{1+ \eps}$. Theorem \ref{thm:main} will be a direct consequence of the forthcoming Theorem \ref{thm:mainCR} and Theorem \ref{thm:mainCC} which give respectively the correctness of the algorithm and its running time. Throughout this section, we take $k = \lceil c / \eps \rceil$, where $c = 200$ is as in  Theorem \ref{thm:k_bound} and $n$ large enough so that $2 (n-2k)^2 \geq n^2$.

Consider the original planted assembly of the puzzle. In this assembly, we refer to pieces located in $[k+1, n-k] \times [k+1, n - k]$ as {\em core} pieces, and to other pieces as {\em peripheral} pieces. We further partition the periphery into $k$ concentric {\em shells}, where shell~$k$ contains those pieces on the boundary of the puzzle, and shell $i$ for $1 \le i \le k-1$ containing those pieces at distance $k - i$ from shell $k$. Shell $0$ is defined similarly, it is the inner boundary of the core. An edge is a peripheral edge if it is adjacent to a peripheral piece. 
A jig of  a piece refers to an edge adjacent to a piece. 

Recall that, for any piece $v$,  if $f$ is in $S_k (v)$ the four pieces $f (\alpha)$, $\alpha \in \{ (0 ,\pm 1), (\pm 1,0) \}$, is called a candidate neighborhood of $v$. Let $c' >0$ be a fixed constant. We say that a puzzle is {\em  typical} if the following properties hold,

\begin{enumerate}[(i)]

\item \label{typ1}Every core piece $v$ has a unique candidate neighborhood.

\item \label{typ2} Every peripheral piece $v$ either has no candidate neighborhood or a unique candidate neighborhood. In this last case, this candidate neighborhood is the neighborhood of the piece in the planted assembly. 
    

\item \label{typ4} The number of peripheral edges with a non-unique color among the peripheral edges is at most $n - 2 k -1$.
    
\item \label{typ5} For every peripheral piece $v$ and two jigs of $v$ (say $j_1$ and $j_2$), no other peripheral piece $u$ has two jigs (say $j_3$ and $j_4$) with matching colors. Namely, $\sigma(j_1) = \sigma(j_3)$ and $\sigma(j_2) = \sigma(j_4)$ cannot hold simultaneously.

\item \label{typ6} For every two colors $a,b$ there are at most $c' k $ pieces with two jigs with these colors.

\end{enumerate}

\begin{lemma}
If $k$ is as above and $c' = 4 / c = 1/50$ in property \eqref{typ6}, with high probability, a random puzzle is typical.
\end{lemma}

\begin{proof}
The first two properties  are a consequence of Theorem \ref{thm:k_bound}. Indeed, from the union bound, Theorem \ref{thm:k_bound} implies that with high probability, for any piece $v \in [n]^ 2$ if $f \in S_k (v)$ then $f(\alpha) = v + \alpha$ for all $\alpha \in \{ (0 ,\pm 1), (\pm 1,0) \}$.  Let us call $E$, the latter event. By definition, if $E$ holds, any piece has at most one candidate neighborhood and this candidate neighborhood is the neighborhood of the piece in the planted assembly. However, if $v$ is a core piece, $S_k(v)$ is non-empty, hence, if $E$ holds, $v$ has necessary a unique candidate neighborhood. This implies properties \eqref{typ1}-\eqref{typ2}.

We check property \eqref{typ4}. Let $J = \Theta ( n k)$ be the number of  peripheral edges and let $m$ be the number of peripheral edges which have a non-unique color among the peripheral edges.
The probability that two different edges have the same color is $1/ q$. Hence, the expectation of $m$ is at most $J (J-1) /  q = O ( (nk)^2 / q)$. Since $q \gg n$, from Markov inequality, it implies that with high probability, $ m = o(n)$. 

We check property \eqref{typ5}. Let us say that pieces $v$ and $u$ have two colors in common,  if we can find two jigs of $v$ (say $j_1$ and $j_2$), and two jigs  of $u$ (say $j_3$ and $j_4$) such that  $\sigma(j_1) = \sigma(j_3)$ and $\sigma(j_2) = \sigma(j_4)$. The probability that two distinct pieces have two colors in common is at most $6^2 / q^{2}$ if these pieces are not adjacent in the planted puzzle and at most $6^2 / q$ if they are adjacent. Hence, the expected number of pairs of peripheral pieces which have two colors in common is at most $ O ( J^2  / q^2 + J / q)$. Since $q \gg n$, the latter is $o(1)$, implying property \eqref{typ5}.

We finally check property \eqref{typ6}. It suffices to prove the claim for pieces whose location $(i,j) \in [n]^2$ satisfies that $i+j$ is odd (even) with $c' /2$ instead of $c'$. We restrict ourselves to those pieces.
Note that no two such pieces share any edge. The probability that a specific piece will have two jigs with colors $a,b$ is at most $6 q^{-2}$. Therefore, by independence, the probability that there are at least $ r\geq 1$ pieces with jigs with colors $a,b$ is at most
$
n^{2r} (6q^{-2})^r . 
$
We take the union bound over all $q^2$ pairs of colors. We find that the probability that there is a pair $a,b$ such that there are at least $r$ pieces with jigs with colors $a,b$ is at most 
$$
6^r n^{2r} q^{-2(r-1)} \leq   6^r n^{2r} n^{-2(r-1) ( 1+ \eps)} = 6^r n^{2  - 2 \eps ( r-1)}. 
$$
For any integer $r > 1  + 1/ \eps$, the latter goes to $0$ with $n$.  Since $\eps \geq k/c$, we can choose $r =2 + k / c$. It follows that there at most $(r-1) \leq 2 k /c$ pieces with two jigs of a given colors. Since $c = 200$, it implies property \eqref{typ6}. \end{proof}

\medskip

We now describe a deterministic algorithm that will reconstruct the planted assembly whenever the underlying puzzle is typical. We describe successively each step of the algorithm on a general puzzle and explain how it proceeds on a typical puzzle. We will later explain how to implement it.

\begin{enumerate}[1.]
\item For each puzzle piece $v$, determine whether it has a candidate neighborhood. If there is no candidate neighborhood mark the piece $v$ as peripheral. If there is a unique candidate neighborhood note which pieces are the neighbors of $v$. Finally, if there is a piece with  a non-unique candidate neighborhood, the algorithm stops here and fails to reconstruct the planted assembly. 
  \end{enumerate}

The properties \eqref{typ1}-\eqref{typ2} imply that the algorithm will not stop for a typical puzzle. Observe also that property \eqref{typ1} implies that all pieces marked as peripheral are indeed peripheral pieces. Note however, that for the other pieces, we do not yet know whether they are peripheral or belong to the core.

\begin{enumerate}[2.]
  \item Greedily join pairs of pieces that are neighbors of each other, as long as possible.  If the largest connected component does not contain a $n - 2 k$ by $n - 2 k$ square, the algorithm stops and fails.  
 \end{enumerate}

For a typical puzzle,  property \eqref{typ1} implies that all core pieces will belong to the same connected component. The condition $2 (n-2k)^2 \geq n^2$ implies that the largest connected component does necessarily contain the core.  Hence the algorithm will not stop here. Importantly, properties \eqref{typ1}-\eqref{typ2} imply that the pieces are necessarily assembled as in the planted assembly.

\begin{enumerate}[3.]
    
\item  From the largest connected component, determine the boundaries of the core (if only one $n-2k$ by $n-2k$ square fits), or guess the boundaries of the core if there is more than one option. (There are at most $2k$ options for where to place the left boundary and at most $2k$ options for where to place the bottom boundary, so altogether there at most $O(k^2)$ possibilities and all of them can be tried.) For simplicity of the presentation, once the core has been determined, disassemble all peripheral pieces and keep only the core.
 \end{enumerate}

For a typical puzzle,  we will have to check that if the guess of the core was not correct then the remainder steps of the algorithm will detect it. On the contrary, if the guess was correct, then the algorithm should return the planted assembly. 

\begin{enumerate}[4.]

\item Greedily assemble the shells of the periphery one by one, from the core towards the inner boundary as follows. Shell $0$ is already assembled. For $0 \leq i \leq k-1$, suppose that shell $i$ was already assembled. To assemble shell $i+1$ find in each one of the four sides of shell $i$ one piece whose free edge (leading out of the assembled part) has a color that appears only once among the yet unassembled peripheral pieces. If no such edge exist for a side, the algorithm is stuck and moves to the next step. Otherwise, find the unique yet unassembled peripheral piece that has an edge of the desired color and insert it in its location. Thereafter, the rest of shell $i+1$ is greedily assembled as follows. Consider an undetermined location next to an already assembled piece of shell $i+1$ which is not one of the four corners of shell $i+1$. This undetermined location is neigbhor of two already assembled pieces, thus it specifies two free edges.  If, among the yet unassembled pieces, there is a unique piece which has matching colors with these two free edges, we insert it here. If not, the algorithm is stuck and moves to the next step. When, all but the four corners of shell $i+1$ are assembled, the above procedure is applied to the four corners. 
 \end{enumerate}
 
Assume that the puzzle is typical and that the guess of the core was correct. We should check that the algorithm finds the planted assembly. We prove by recursion on $i$, $0 \leq i \leq k-1$, that the algorithm reconstructs correctly shell $i+1$. To this end, notice that property \eqref{typ4} implies that for each side of shell $i$, $0 \leq i \leq k-1$, there will be at least one free edge among the $n - 2k - 2i$ free edges with a color which appears once among the yet unassembled pieces. Then, thanks to property \eqref{typ5}, we will reconstruct unambiguously shell $i+1$.

Assume that the puzzle is typical and the guess of the core was not correct.  We should check that the algorithm is stuck at some point. As pointed earlier,  the guessed core is an $n-2k$ by $n-2k$ square in the planted assembly. If the algorithm has not been stuck earlier, it will reconstruct successive shells until one side of length $n - 2 k +2 i$ of the assembled pieces is on the boundary of the planted assembly for some $0 \leq i \leq k-1$. Then,  by property \eqref{typ4} at least one of the free edges  on this side has a color which is not present among the yet unassembled pieces. Hence, it will not be possible to assemble it and the algorithm will be stuck.

\begin{enumerate}[5.]
   
\item If a properly colored assembly has been found, the algorithm returns this assembly. Otherwise, try a new guess for the core and repeat stage 4. If all guesses for the core have been tried, the algorithm stops and fails. 
    
\end{enumerate}

The above analysis of the algorithm has proved its correctness on typical puzzles.  (Note that we have not used so far the property \eqref{typ6}.) 

\begin{theorem}\label{thm:mainCR}
If the puzzle is typical then the above algorithm recovers the planted puzzle. 
\end{theorem}

We now analyze the complexity of the algorithm, this is where property \eqref{typ6} will be used.

\begin{theorem}\label{thm:mainCC}
If the puzzle is typical then the above algorithm can be implemented to run in time $O(k) ^{k^2} n^{O(k)}$. 
\end{theorem}

\begin{proof}
There are most $  O ( \min( n^2, q )^2) = n ^{O(1)}$ pairs of colors used in the puzzle. In time $n^{O(1)}$, we can build a table which to any such pair of colors returns the set of pieces which have matching colors. Property \eqref{typ6} implies that  for all pairs of colors this set has cardinal at most $ m = c' k  = O(k)$.

We perform step 1 of the algorithm by listing all the feasible assemblies of $[-k,k]^2$. This list can be computed in time $O(k) ^{k^2} n^{O(k)}$ in the following manner:
\begin{enumerate}[(a)]
\item Enumerate over all possible pieces in the top row and left column of the square. That is, we enumerate all local assembly on $W = (\{-k\} \times [-k,k]) \cup ( [-k,k] \times \{k\})$. 
\item For each feasible local assembly on $W$, we enumerate all pieces that can be placed on the top and left corner of $[-k,k]^2 \backslash W$. It gives the set of feasible assembly on $W' =  W \cup \{ (- k +1, k-1 ) \}$. 
\item We repeat the previous step to $W'$ and proceed sequentially from top to bottom and left to right. 
\end{enumerate}

The output of the algorithm is the enumeration of all feasible assembly of $[-k,k]^2$.  By exhaustive search, the running time of part (a) is $(n^{2} )^{2k} = n^{O(k)}$. For the part (b)-(c), the running time to enumerate all feasible assembly whose restriction to $W$ is fixed is $O(m ^{k^2}) = O(k)^{k^2}$ where $m$ is as above. It corresponds to the  calls in the table which to any pair of colors return the set pieces which have matching colors. Indeed, once the top row and left column are fixed, each new piece has two colors constrained. There are at most $ m^{(k-1)^2}$ calls in this table.

In the process of computing this list of all feasible assemblies,  when a new feasible assembly on $[-k,k]^2$ is found, we update in time $O(1)$, the candidate neighborhood of the central piece.  It follows that step $1$ of the algorithm can be performed in time  $O(k)^{k^2} n^{O(k)}$.

Step 2 is performed in time $O(n^2)$ by a greedy exploration.  The choice of possible cores in step 3 will require at most $O(k^2)$ trials of the remainder steps. In step 4, to reconstruct shell $i+1$,  it first requires a time $O(k n ^2)$ to find on each side, the free edge with unique color. Then, the  reconstruction of the shell will require a time $O(n)$, corresponding to the $ 4 (n - 2k - 2i) $ calls in the table which to any pair of colors return the set pieces which have matching colors. We obtain the claimed running time for the algorithm. \end{proof}

\section{Variants}
\label{sec:fin}

The model that we have studied can be generalized to a model where the jigs have a shape and the pieces are allowed to be rotated. This could be formalized using (oriented) edges as follows. The set of edges $e = (x,y)$ of the grid such that $x \in [n]^2$ is denoted by $E$. The set $E^{\INT}$ is the subset of edges such that both $x$ and $y$ are in $[n]^2$.  It is stable under the involution $ \check \cdot$ defined for every $e = (x,y)$ by $ \check e = (y,x)$. The edges  adjacent to $x \in \dZ^2$ are organized  in counter-clockwise order (right, up, left and down), we set
\begin{equation*}\label{eq:defEx}
x + B = \PAR{ (x,x+(1,0)), (x,x+(0,1)), (x,x-(1,0)), (x,x - (0,1)) }.
\end{equation*}
where  $B= \PAR{ (1,0), (-1,0), (0,1) , (0,-1) }$.  Now, each edge receives a {\em jig} according to a function $\sigma :   E  \to [q]$. The set of jigs $[q]$ is equipped with an involution $\iota : [q] \to [q]$. We interpret $\iota (j) = j'$ as the jigs $j$ and $j'$ match together, see Figure \ref{fig1}. A {\em puzzle} is a then function $\sigma$ such that for all $e \in   E^{\INT}$, 
\begin{equation}\label{eq:puzzle}
\sigma(e) = \iota ( \sigma (  \check e) ).
\end{equation}
The case that we have treated previously corresponds to $\iota$ equal to the identity. 
\begin{figure}[htb]
\centering \scalebox{0.3}{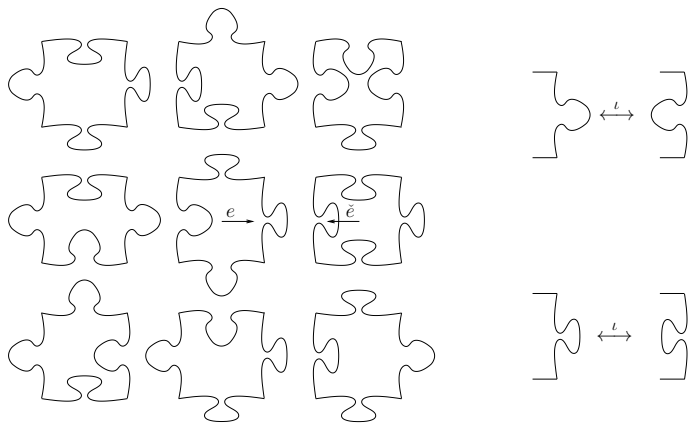}
\caption{A puzzle with $n = 3$, $q =4$ and the involution $\iota$.}\label{fig1}
\end{figure}

We now define the way the pieces can be assembled. The {\em cyclic group} $C_4 \subset S_4$ is the subgroup of permutations generated by $ (1 \, 2 \, 3 \, 4)$. Below, if $a$ is a function on $  E$, $s \in S_4$ and $F = (f_1, \cdots , f_4) \in   E^4$, we set $a(F) = (a (f_1), \cdots , a(f_4))$ and $F_s = (f_{s(1)}, \cdots , f_{s(4)})$. An {\em assembly} $a$ is a permutation on $  E$ which satisfies : 
\begin{enumerate}[(i)]
\item  for every $x \in [n]^2$, there exist a piece $y \in [n]^2$ and $c \in C_4$, such that $a (x + B) = y + B_c$, 
\item
if $ y= (1,1)$, the above permutation $c$ is the identity. 
\end{enumerate}
In words, condition (i) says that piece $y$ is assigned a location $ x \in [n]^2$ and is rotated by an angle multiple of $\pi/2$. By construction the map which to $x$ assigns $y$ is a bijection of $[n]^2$. Condition (ii) fixes a global orientation to the puzzle. We will say that an assembly is {\em feasible} if for all $e \in   E^{\INT}$, 
$$
\sigma ( a (e) ) = \iota  \PAR{ \sigma  ( a ( \check e ) }.  
$$

A feasible assembly is a solution of the puzzle : all pieces are in a position where the jigs match. Note that by definition, the identity is a feasible assembly : it gives back the pieces in their original position.   
We say that a puzzle has {\em unique vertex assembly} if it has only one feasible assembly (note that without condition (ii), it would only be possible to uniquely assemble the puzzle up to  a global rotation by a multiple of  $\pi/2$). 

Observe that, unlike in a usual jigsaw puzzle, the boundary pieces (pieces in $[n]^2 \backslash [2,n-1]^2$) cannot be distinguished from the other pieces. To recover a usual jigsaw puzzle, we may simply consider the subset of assembly which satisfy the extra condition $a(E^ {\INT} ) = E^ {\INT}$ (so that edges on the boundary remain on the boundary).

In this new setting, a {\em random puzzle} is simply obtained by sampling the function $\sigma$ uniformly  on the set of puzzles (functions $\sigma$ which satisfies \eqref{eq:puzzle}).  Hence, up to the constraint \eqref{eq:puzzle}, the jigs are independent and uniformly distributed. Theorem \ref{thm:main} continues to hold on this extended setting. Indeed, it is easy to check that the proof of Theorem \ref{thm:main} continues to work if we adapt the definition of the constraint graph (to accommodate the involution).

\end{document}

%% file: puzzle2.pdf_t
\begin{picture}(0,0)%
\includegraphics{puzzle2.pdf}%
\end{picture}%
\setlength{\unitlength}{4144sp}%
\begingroup\makeatletter\ifx\SetFigFont\undefined%
\gdef\SetFigFont#1#2#3#4#5{%
  \reset@font\fontsize{#1}{#2pt}%
  \fontfamily{#3}\fontseries{#4}\fontshape{#5}%
  \selectfont}%
\fi\endgroup%
\begin{picture}(7985,1914)(17809,-7408)
\put(17911,-6046){\makebox(0,0)[lb]{\smash{{\SetFigFont{12}{14.4}{\familydefault}{\mddefault}{\updefault}{\color[rgb]{0,0,0} {\Huge $(3,2)$}}%
}}}}
\put(18811,-6046){\makebox(0,0)[lb]{\smash{{\SetFigFont{12}{14.4}{\familydefault}{\mddefault}{\updefault}{\color[rgb]{0,0,0} {\Huge $(1,2)$}}%
}}}}
\put(17911,-6991){\makebox(0,0)[lb]{\smash{{\SetFigFont{12}{14.4}{\familydefault}{\mddefault}{\updefault}{\color[rgb]{0,0,0} {\Huge $(1,1)$}}%
}}}}
\put(18811,-6991){\makebox(0,0)[lb]{\smash{{\SetFigFont{12}{14.4}{\familydefault}{\mddefault}{\updefault}{\color[rgb]{0,0,0} {\Huge $(3,1)$}}%
}}}}
\end{picture}%

%% file: puzzle11.pdf_t
\begin{picture}(0,0)%
\includegraphics{puzzle11.pdf}%
\end{picture}%
\setlength{\unitlength}{4144sp}%
\begingroup\makeatletter\ifx\SetFigFont\undefined%
\gdef\SetFigFont#1#2#3#4#5{%
  \reset@font\fontsize{#1}{#2pt}%
  \fontfamily{#3}\fontseries{#4}\fontshape{#5}%
  \selectfont}%
\fi\endgroup%
\begin{picture}(13554,8270)(23007,-11403)
\put(34741,-9781){\makebox(0,0)[lb]{\smash{{\SetFigFont{17}{20.4}{\familydefault}{\mddefault}{\updefault}{\color[rgb]{0,0,0}{\Huge $\stackrel{\iota}{\longleftrightarrow}$}}%
}}}}
\put(34786,-5371){\makebox(0,0)[lb]{\smash{{\SetFigFont{17}{20.4}{\familydefault}{\mddefault}{\updefault}{\color[rgb]{0,0,0}{\Huge $\stackrel{\iota}{\longleftrightarrow}$}}%
}}}}
\put(29746,-7306){\makebox(0,0)[lb]{\smash{{\SetFigFont{12}{14.4}{\familydefault}{\mddefault}{\updefault}{\color[rgb]{0,0,0}{\Huge $\check e$}}%
}}}}
\put(27361,-7306){\makebox(0,0)[lb]{\smash{{\SetFigFont{12}{14.4}{\familydefault}{\mddefault}{\updefault}{\color[rgb]{0,0,0}{\Huge $e$}}%
}}}}
\end{picture}%